\documentclass[12pt]{amsart} 
\usepackage{amscd}
\usepackage{amssymb}
\usepackage{a4wide}
\usepackage{amstext}
\usepackage{amsthm}
\usepackage{xcolor}
\numberwithin{equation}{section}

\DeclareMathOperator{\im}{Im}

\renewcommand{\phi}{\varphi}

\newcommand{\pw}{\mathcal{P}W_\pi}

\newcommand{\rl}{\mathbb{R}}

\newcommand{\he}{\mathcal{H}(E)}

\newcommand{\closspan}{\overline{\rm Span}\,}

\newtheorem{Thm}{Theorem}[section]
\newtheorem{theorem}[Thm]{Theorem}

\newtheorem{question}{Question}

\newtheorem{proposition}[Thm]{Proposition}

\begin{document}
\sloppy
\title[Complementability of exponential systems]{Complementability of exponential systems}
\author{Yurii Belov}
\address{Yurii Belov,
\newline Chebyshev Laboratory, St.~Petersburg State University, St. Petersburg, Russia,
\newline {\tt j\_b\_juri\_belov@mail.ru}
}
\thanks{Author was supported by RNF grant 14-21-00035.}

\begin{abstract} We prove that any incomplete system of complex exponentials $\{e^{i\lambda_n t}\}$ in $L^2(-\pi,\pi)$ is a subset of some complete and minimal system of exponentials. In addition, we prove analogous statement for systems of reproducing kernels in de Branges spaces.
\end{abstract} 


\maketitle

\section{Introduction}

 Let a sequence $\Lambda\subset\mathbb{C}$ be such that the system of exponential functions $\{e^{i\lambda t}\}_{\lambda\in\Lambda}$ is complete and minimal in $L^2(-\pi,\pi)$. Numerous papers  are devoted to the different aspects of theory of exponential systems such as completeness, Riesz basis  property,
Riesz sequence property, linear summation methods e.t.c. We refer to \cite{Red, young2, Koo1, HNP} for the details.
The geometry of such sequences is not well-undrestood. Nevertheless there exist a non-geometric necessary and sufficient conditions for the sequence $\{e^{i\lambda t}\}_{\lambda\in\Lambda}$ to be complete and minimal; see Theorem 4 in Lecture 18 of \cite{Lev}.

The systems of exponentials are much more regular than an arbitrary system of vectors in $L^2(-\pi,\pi)$. For example, it is easy to verify that  if $e^{iat}\in\closspan\{e^{i\lambda t}\}_{\lambda\in\Lambda}$, $a\notin\Lambda$, then $\closspan\{e^{i\lambda t}\}_{\lambda\in\Lambda}=L^2(-\pi,\pi)$.

The system $\{e^{i\lambda t}\}_{\lambda\in\Lambda}$ is said to be a Riesz basis if there exists an isomorphism
$T : L^2(-\pi,\pi)\rightarrow L^2(-\pi,\pi)$
such that $T(e^{int})=e^{i\lambda_n t}\cdot\|e^{i\lambda_n t}\|^{-1}$, $n\in\mathbb{Z}$ (or, which is the same, $\{e^{i\lambda t}\}_{\lambda\in\Lambda}$ is an image of an orthogonal basis under a bounded invertible operator). Moreover, $\{e^{i\lambda t}\}_{\lambda\in\Lambda}$ is said to be a Riesz sequence if it is a Riesz basis of the closure of the space spanned by elements from $\{e^{i\lambda t}\}_{\lambda\in\Lambda}$.
 
The next question was raised in \cite[p. 196]{Nak}:

\medskip

{\bf Question.} {\it Can every Riesz sequence of complex exponentials be complemented
up to a complete and minimal system of complex exponentials?}

\medskip

K. Seip has shown in \cite[Theorem 2.8]{S} that system $\{e^{\pm i(n+\sqrt{n})t}\}_{n>1}$ is a Riesz sequence of complex
exponentials in $L^2(-\pi,\pi)$ which cannot be complemented up to a Riesz basis. We will give a positive answer to Question. Moreover, we don't need Riesz sequence assumption.

\begin{theorem}
If $\{e^{i\lambda t}\}_{\lambda\in\Lambda}$ is an incomplete system in $L^2(-\pi,\pi)$, then there exists a sequence $S\subset\mathbb{R}$, $\Lambda\cap S=\emptyset$ such that the system $\{e^{i\lambda t}\}_{\lambda\in\Lambda\cup S}$ is complete and minimal in $L^2(-\pi,\pi)$.
\label{mainth}
\end{theorem}

In \cite{Nak} this was proved only for very specific real sequences $\Lambda$ ($\Lambda=\{n+\delta_n\}_{n\in\mathbb{Z}}$, $\delta_n$ is an increasing sequence). It is interesting to note that there exist a lot of complete exponential systems with no complete and minimal subsequences. The simplest example is  any sequence $\{\lambda_n\}$ such that $|\lambda_n|\rightarrow0$ (but it may happen even if $|\lambda_n|\rightarrow\infty$; see Section \ref{s4}).

The proof of Theorem \ref{mainth} is nonconstructive and based on deep and significant de Branges theory of entire functions. In $\S3$ we will prove the analogous theorem for the systems of reproducing kernels in {\it an arbitrary de Branges space}. It would be interesting to find a direct and constructive proof of Theorem \ref{mainth}.

In the next section we transfer our problem to the Paley--Wiener space of entire functions.  The Subsection 3.1 is devoted to the de Branges theory. The Subsection 3.2 is devoted to the proof of our result.

\section{Paley--Wiener space}

We reformulate our problem in the Paley-Wiener space 
$$\pw:=\{F: F(z)=\frac{1}{2\pi}\int_{-\pi}^{\pi}g(t)e^{itz}dt,\quad g\in L^2(-\pi,\pi)\}.$$
The Fourier transform $\mathcal{F}$ acts unitarily from $L^2(-\pi,\pi)$ onto $\pw$, under this transform the exponential functions become the reproducing kernels in $\pw$,
$$\mathcal{F} \left ( e^{-i\bar{\lambda}t} \right ) =\frac{\sin \pi(z-\bar{\lambda})}{\pi(z-\bar{\lambda})}=:k^{\pw}_\lambda(z),$$
$$(F,k^{\pw}_\lambda)_{\pw}=F(\lambda).$$
This gives us the following description (see e.g. Theorem 4 in Lecture 18 of \cite{Lev}).
\begin{proposition} The system $\{e^{i\lambda t}\}_{\lambda\in\Lambda}$ is complete and minimal in $L^2(-\pi,\pi)$ if and only if  
the following three conditions hold
\begin{enumerate}
\begin{item}
the product 
\begin{equation}
G(z):=\lim_{R\rightarrow\infty}\prod_{|\lambda|< R}\biggl{(}1-\frac{z}{\lambda}\biggr{)}
\label{prod}
\end{equation}
converges to an entire function $G$;
\end{item}
\begin{item}
For some (any) $\lambda\in \Lambda$ we have $\frac{G(z)}{z-\lambda}\in\pw$;
\end{item}
\begin{item}
There is no non-zero entire $T$ such that $GT\in\pw$.
\end{item}
\end{enumerate}
\end{proposition}
If the system  $\{e^{i\lambda t}\}_{\lambda\in\Lambda}$ is incomplete, then the product \eqref{prod} may diverge. But since $\Lambda$ satisfies Blaschke condition $\sum_{\lambda\in\Lambda}\frac{|\im \lambda|}{|\lambda|^2}<\infty$  the product
\begin{equation}
G_\Lambda(z):=\prod_{\lambda\in\Lambda}\biggl{(}1-\frac{z}{\lambda}\biggr{)}e^{\Re(\lambda^{-1})z}
\label{prodB}
\end{equation}
converges. Function $G_\Lambda$ is such that $\dfrac{G(z)}{\overline{G(\bar{z})}}$ is  {\it a ratio of two Blaschke products} in $\mathbb{C}^+$.

Let  $\{e^{i\lambda t}\}_{\lambda\in\Lambda}$ be an incomplete system in $L^2(-\pi,\pi)$ and $G_\Lambda$ be a canonical product from \eqref{prodB}. Put
$$\mathcal{H}_{\Lambda,\pi}:=\{T: T - \text{  entire, and } G_\Lambda T\in\pw\}.$$
It is easy to verify that $\mathcal{H}_{\Lambda,\pi}$ is a nontrivial Hilbert space of entire functions (with respect to the norm inheritted from Paley-Wiener space).
\begin{proposition}
Let  $\{e^{i\lambda t}\}_{\lambda\in\Lambda}$ be an incomplete system. 
Then the system  $\{e^{i\lambda t}\}_{\lambda\in\Lambda\cup S}$ is complete and minimal in $L^2(-\pi,\pi)$ if and only if the set $S$ is a minimal uniqueness set in $\mathcal{H}_{\Lambda,\pi}$ (i.e  $S$ is a uniqeness set but $S\setminus\{s_0\}$ is not uniqueness set for some (any) $s_0\in S$).
\label{prop1}
\end{proposition} 
So, it remains to prove that in $\mathcal{H}_{\Lambda,\pi}$ there exists a minimal uniqueness set.

\section{De Branges spaces of entire functions}
\subsection{Preliminary steps} It is well known that $\pw=\mathcal{E}_\pi\cap L^2(\mathbb{R})$, where $\mathcal{E}_\pi$ is a class of entire functions of exponential type at most $\pi$. In other words, we have {\it global growth condition} and {\it inclusion onto some weighted Hilbert space on the real axis}. The key idea is that the space $\mathcal{H}_{\Lambda,\pi}$ is of the same kind. 

Informally speaking, the de Branges spaces form a general class of spaces of such type. Now we remind two possible ways to define de Branges spaces.

1. An entire function $E$ is said to be in the Hermite--Biehler class
if $|E(z)| >|E^*(z)|$,  
$z\in {\mathbb{C}_+}$, where $E^* (z) = \overline {E
(\overline z)}$. With any such function we associate the {\it de
Branges space} $\mathcal{H} (E) $ which consists of all entire
functions $F$ such that $F/E$ and $F^*/E$ restricted to
$\mathbb{C_+}$ belong to the Hardy space $H^2=H^2(\mathbb{C_+})$.
The inner product in $\he$ is given by
$$
(F,G)_{\he} = \int_\rl \frac{F(t)\overline{G(t)}}{|E(t)|^2} \,dt.
$$
2. We will say that $\mathcal{H}$ is a de Branges space if $\mathcal{H}$ is a nontrivial Hilbert space, whose elements are entire functions, which satisfies the following axioms:

\begin{itemize}
\item[(A1)] Whenever $F$ is in the space and has a nonreal zero $w$, the function $F(z)\frac{z-\bar{w}}{z-w}$ is in the space and has the same norm as $F$;
\item[(A2)] For every nonreal number $w$, the evaluation functional $F\mapsto F(w)$ is continuous;
\item[(A3)] The function $F^*$ belongs to the space whenever $F$ belongs to the space and it always has the same norm as $F$.
\end{itemize}

These two definitions are equivalent (see \cite[Theorem 23]{br}). Sometimes it is natural to require additional assumption that for any $w\in\mathbb{C}$ there exists  $F\in\mathcal{H}$ such that $F(w)\neq0$. This corresponds to the situation when $E$ has no real zeros. In what follows we will consider only such spaces.

Any de Branges space has reproducing kernel which is given by 
\begin{equation}
\label{repr}
k_w(z)=\frac{\overline{E(w)} E(z) - \overline{E^*(w)} E^*(z)}
{2\pi i(\overline w-z)} =
\frac{\overline{A(w)} B(z) -\overline{B(w)}A(z)}{\pi(z-\overline w)},
\end{equation}
where the entire functions $A$ and $B$ are defined by $A = \frac{E+E^*}{2}$,
$B=\frac{E^*-E}{2i}$. Functions $A$ and $B$ have only simple real zeros, 
and $E=A - iB$.

If $E(z)=e^{-i\pi z}=\cos(\pi z) - i\sin(\pi z)$, then $\mathcal{H}(E)=\pw$.

Any de Branges space has orthogonal basis which consists of reproducing kernels. More precisely,

\medskip

\cite[Theorem 22]{br}: {\it Let $\mathcal{H}(E)$ be a given space and let $\phi(t):=-\arg E(t)$ be a phase function associated with $E(z)$. If $\alpha$ is a given real number, the system $\{k_{t_n}(z)\}_{\phi(t_n)-\alpha\in\pi\mathbb{Z}}$ is an orthonormal basis for all $\alpha\in[0,\pi)$ except may be one.}

\medskip

For simplicity we will assume that system $\{k_{t_n}(z)\}$ is an orthogonal basis for $\alpha=0$. So, $\{t_n\}=\{z:A(z)=0\}$.

It is interesting to note that if a Hilbert space of entire functions contains {\it two orthogonal bases of reproducing kernels}  then 
it is equal to some de Branges space of entire functions (see \cite{BMS1}).

 In particular,  in any de Branges space there exists a minimal uniqueness set $\{x: A(x)=0\}$.

Further we will use the following simple fact about de Branges spaces which follows from the definition of $\he$.
\begin{proposition} Let $F$ belong to $\he$ and $B$ be a Blaschke product such that $FB$ is an entire function. Then $FB\in\he$.
\label{Bl}
\end{proposition}

\subsection{Proof of Theorem \ref{mainth}} We claim that $\mathcal{H}_{\Lambda,\pi}$ is  a de Branges space (there exists $E$ such that $\he=\mathcal{H}_{\Lambda,\pi}$). It is easy to see that for any $w\in\mathbb{C}$ there exists $F\in\mathcal{H}_{\Lambda,\pi}$ such that $F(w)\neq 0$. Without loss of generality we can assume that $\{z:A(z)=0\}\cap\Lambda=\emptyset$ (we can always choose another $\alpha$ in \cite[Theorem 22]{br}). Put $S=\{z:A(z)=0\}$.
From Proposition \ref{prop1} we get that system $\{e^{i\lambda t}\}_{\lambda\in\Lambda\cup S}$ is a complete and minimal system in $L^2(-\pi,\pi)$. 

To prove the claim we will use axiomatic description of de Branges spaces. We need only to verify that $\mathcal{H}_{\Lambda,\pi}$
satisfies axioms (A1)--(A3). Let $G_\Lambda$ be the canonical product from \eqref{prodB}.

(A2). For any $w\in\mathbb{C}$ the point evaluation functional $F\mapsto F(w)$ is a bounded linear functional on $\mathcal{H}_{\Lambda,\pi}$. Indeed, if $w\notin\Lambda$, then
$$|F(w)|=|G_\Lambda(w)F(w)\slash G_\Lambda(w)|\leq|G_\Lambda(w)|^{-1}\|G_\Lambda F\|_{\pw}\cdot\|k^{\pw}_w\|$$
$$=|G_\Lambda(w)|^{-1}\|F\|_{\mathcal{H}_{\Lambda,\pi}}\cdot\|k^{\pw}_w\|.$$ 
If $w\in\Lambda$, then we can write $F(z)=\frac{1}{2\pi}\int_0^{2\pi}F(w+\varepsilon e^{i\theta})d\theta$, and use the same estimate.

(A1). If $F\in\mathcal{H}_{\Lambda,\pi}$ and $F(w)=0$, then $F(z)\frac{z-\bar{w}}{z-w}\in\mathcal{H}_{\Lambda,\pi}$ and
$\|F(z)\frac{z-\bar{w}}{z-w}\|=\|F\|$.

(A3). If $F\in\mathcal{H}_{\Lambda,\pi}$, then $F^*\in\mathcal{H}_{\Lambda,\pi}$. Indeed, if $G_\Lambda F\in\pw$, then $G_\Lambda^*F^*\in\pw$. Using Proposition \ref{Bl} and the fact that $\dfrac{G_\Lambda}{G^*_\Lambda}$ is a ratio of two Blaschke products we get that $G_\Lambda F^*\in\pw$. \qed

\subsection{Reproducing kernels in de Branges spaces} Theorem \ref{mainth} can be generalized to an arbitrary de Branges space.
\begin{theorem}
Let $\he$ be a de Branges space. If $\{k_\lambda\}_{\lambda\in\Lambda_1}$ is an incomplete system of
reproducing kernels in $\he$, then there exists a system $\{k_\lambda\}_{\lambda\in\Lambda_2}$ such that the system $\{k_\lambda\}_{\lambda\in\Lambda_1\cup\Lambda_2}$ is complete and minimal in $\he$.
\label{deBr}
\end{theorem}
\begin{proof}
Let 
$$\mathcal{H}_{\Lambda_1,\he}:=\{T: T - \text{ entire and } G_{\Lambda_1} T\in\he\}.$$
Now we can repeat the arguments from the previous subsection.
\end{proof}
Theorem \ref{deBr} can be useful for the research of strong M-basis property in de Branges spaces (see \cite[Section 8]{BBB2}).

\section{Concluding remarks and open questions\label{s4}}

1. It is well known that there is no $F\in\pw\setminus\{0\}$ which vanishes only on imaginary axis. On the other hand, the set $\{z:F(z)=0\}$ has to satisfy the Blaschke condition. Hence, the system $\{e^{nt}\}_{n\in\mathbb{Z}}$ is a complete set in $L^2(-\pi,\pi)$ which contains no complete and minimal subset.

2. It is interesting to know whether Theorem \ref{deBr} true for the Fock space of entire functions.

3. It is not clear whether it is possible to keep some regularity of the system of reproducing kernels when we complement our system.
\begin{question}
Is it true that any incomplete system of reproducing kernels with infinite defect $dim(L^2(-\pi,\pi)\ominus\closspan\{e^{i\lambda t}\}_{\lambda\in\Lambda})=\infty$ can be complemented up to a strong M-basis?
\end{question}
For a detailed review about strong M-basis property for system of complex exponentials in $L^2(-\pi,\pi)$ see \cite{BBB2} and references therein.

\end{document}